\documentclass[11pt]{article}
\usepackage[a4paper,margin=1in]{geometry}
\usepackage[T1]{fontenc}
\usepackage[utf8]{inputenc}
\usepackage{lmodern}
\usepackage{microtype}
\usepackage{amsmath,amssymb,amsthm,mathtools}
\usepackage{booktabs,longtable,array,tabularx}
\usepackage{enumitem}
\usepackage{placeins}
\usepackage{tikz}
\usetikzlibrary{arrows.meta,positioning,calc,fit}
\usepackage{xcolor}
\usepackage{hyperref}
\hypersetup{colorlinks=true,linkcolor=blue!50!black,citecolor=blue!50!black,urlcolor=blue!50!black,%
  pdftitle={The crumby coloring conjecture for subcubic outerplanar graphs},%
  pdfauthor={Jozsef Pinter}}
\definecolor{BlueV}{RGB}{30,90,180}
\definecolor{RedV}{RGB}{190,45,45}
\tikzset{
  vblue/.style={circle,draw=BlueV,fill=BlueV!18,thick,inner sep=1.8pt,minimum size=6.5mm},
  vred/.style={circle,draw=RedV,fill=RedV!18,thick,inner sep=1.8pt,minimum size=6.5mm},
  vgray/.style={circle,draw=black!65,fill=black!8,thick,inner sep=1.7pt,minimum size=6.3mm},
  edge/.style={thick},
  chord/.style={thick,dashed},
  proofnode/.style={draw,rounded corners,align=center,minimum height=9mm,inner xsep=5pt,fill=black!3},
}
\newtheorem{theorem}{Theorem}[section]
\newtheorem{lemma}[theorem]{Lemma}
\newtheorem{proposition}[theorem]{Proposition}

\theoremstyle{definition}
\newtheorem{definition}[theorem]{Definition}

\newcommand{\B}{\mathrm B}
\newcommand{\R}{\mathrm R}
\newcommand{\calB}{\mathcal B}
\newcommand{\calC}{\mathcal C}
\newcommand{\calP}{\mathcal P}
\newcommand{\calR}{\mathcal R}
\newcommand{\real}{\operatorname{Re}}
\newcommand{\depth}{\operatorname{depth}}
\newcommand{\Ser}{\operatorname{Ser}}
\newcommand{\Chord}{\operatorname{Chord}}
\newcommand{\Dec}{\operatorname{Dec}}
\newcommand{\Root}{\operatorname{Root}}
\newcommand{\Tree}{\operatorname{Tree}}
\newcommand{\FL}{\mathrm{FL}}
\title{The crumby coloring conjecture for subcubic outerplanar graphs}
\author{%
J\'ozsef Pint\'er\textsuperscript{1,2}\thanks{Corresponding author. Email: \href{mailto:pinterj@edu.bme.hu}{\texttt{pinterj@edu.bme.hu}}.}\\[0.45em]
\small
\begin{tabular}{@{}c@{}}
\textsuperscript{1}Department of Stochastics, Institute of Mathematics,\\
Budapest University of Technology and Economics,\\
Egry J\'ozsef utca 1, 1111 Budapest, Hungary\\[0.25em]
\textsuperscript{2}HUN-REN--BME Stochastics Research Group,\\
Egry J\'ozsef utca 1, 1111 Budapest, Hungary
\end{tabular}}
\date{}
\begin{document}
\maketitle
\begin{abstract}
The red--blue vertex partitions now known as crumby colorings originate in a conjecture of Thomassen related to Wegner's conjecture on squares of planar graphs. In such a coloring, the blue vertices induce a graph of maximum degree at most one, while the red vertices induce a graph with no isolated vertices and no simple path with three edges. Bar\'at, Bl\'azsik and Dam\'asdi proved that every \(2\)-connected outerplanar graph of maximum degree at most three admits a crumby coloring, and conjectured that the \(2\)-connectivity assumption can be removed. We prove this conjecture: every finite simple subcubic outerplanar graph admits a crumby coloring. To prove the conjecture, we introduce a rooted grammar for subcubic outerplanar graphs. The grammar describes such graphs recursively using rooted branches and two-terminal path fragments; cyclic blocks are handled by deleting the root to obtain a path fragment. We correspondingly extend crumby colorings to crumby-admissible colorings: in a rooted branch the root, and in a path fragment the two terminals, are allowed to be temporary isolated red vertices. This relaxation makes induction along the grammar possible while retaining only finite boundary information. The induction reduces to verifying an explicit finite family of lower certificates, namely nonempty sets of boundary types and root states. The required verification has two parts: the family must be closed under all steps of the decomposition, and every certified completed branch must contain a final-legal root state, so that the temporary defect disappears and the resulting coloring is a genuine crumby coloring. This final step is computer-assisted: a stand-alone certificate checker, supplied with the paper, verifies the stated closure and crumby conditions for the supplied certificate. All structural reductions and the certificate-induction principle are proved by hand.
\end{abstract}
\begin{center}
\end{center}
\section{Introduction}\label{sec:intro}
Motivated by Wegner's conjecture \cite{Wegner} on squares of planar graphs, Thomassen conjectured that every \(3\)-connected cubic graph admits a red--blue vertex partition in which the blue vertices induce a graph of maximum degree at most one, while the red vertices induce a graph of minimum degree at least one and containing no path with three edges \cite{Thomassen}.\footnote{There is a small-order obstruction: the triangular prism \(C_3\mathbin{\square}K_2\), the Cartesian product of a \(3\)-cycle and an edge, is a \(3\)-connected cubic graph on six vertices and admits no such partition \cite[Remark~1.3]{barat_decomposition_2019}. In contrast, \(K_4\), \(K_{3,3}\), and the cube \(Q_3\) do admit such partitions. The conjecture is therefore stated for graphs on at least eight vertices \cite[Conjecture~1.2]{barat_decomposition_2019}.} Following Bar\'at \cite{barat_decomposition_2019}, such a partition is called a \emph{crumby coloring}. Bellitto, Klimo\v{s}ov\'a, Merker, Witkowski and Yuditsky showed that Thomassen's conjecture fails for \(3\)-connected cubic graphs in general \cite{BellittoEtAl}, which makes it natural to ask in which graph classes crumby colorings exist.

Bar\'at established the existence of crumby colorings for two families of generalized Petersen graphs and for all subcubic trees \cite{barat_decomposition_2019}. Bar\'at, Bl\'azsik and Dam\'asdi then proved that every \(2\)-connected outerplanar graph of maximum degree at most three admits a crumby coloring \cite[Theorem~8]{BBD}, and conjectured that the \(2\)-connectivity hypothesis is unnecessary: every outerplanar graph of maximum degree at most three should admit a crumby coloring \cite[Conjecture~9]{BBD}. We prove this conjecture in full.
\begin{theorem}[Main theorem]\label{thm:main}
Every finite simple subcubic outerplanar graph admits a crumby coloring.
\end{theorem}
The result is close to best possible. Although every outerplanar graph is \(K_4\)-minor-free, the analogous statement fails already for subcubic partial \(2\)-trees: there exist connected and \(2\)-connected examples with no crumby coloring, on \(18\) and \(40\) vertices respectively \cite{PinterK4}. Thus outerplanarity cannot be replaced merely by treewidth two.

We outline the proof. We work with rooted graphs and root a connected subcubic outerplanar graph at a vertex of degree at most two. At the root there are two possibilities. Either the root meets only bridges, and deleting it leaves at most two smaller rooted pieces; or the root lies on a cyclic block and uses its two boundary-cycle edges, and deleting it opens the block into a two-terminal interval along the outer boundary. Inside such an interval every vertex already spends two incidences on the boundary, so subcubicity leaves at most one further incidence, used either by a chord or by a bridge to a smaller rooted piece, and the chords form a noncrossing matching. Iterating this reduction expresses every rooted subcubic outerplanar graph through a recursive decomposition into rooted branches and two-terminal intervals.

To make induction compatible with cutting and later reassembling the graph, we work with a relaxed form of crumby coloring. A \emph{crumby-admissible coloring} satisfies the usual crumby conditions apart from the active boundary, but allows the root of a rooted branch, or the two terminals of an interval, to be temporary isolated red vertices. These are the only defects permitted, and they are recorded as part of the boundary data.

The coloring constraints then have only finite memory along an interval. What a partial coloring can still contribute to later stages is visible entirely at the two terminals: for a blue terminal, whether it already has a blue neighbor; for a red terminal whose red component avoids the other terminal, the length of a longest red path leaving it, which is \(0\), \(1\), or \(2\); and, when both terminals lie in one red component, which of the six possible ordered connected shapes that component has. This leaves only finitely many boundary types, and they are stable under the decomposition: two intervals with the same set of realizable boundary types behave identically under every step. Hence the interior of an interval may be replaced by its finite set of boundary types, and the existence proof becomes an induction over the decomposition that carries finitely many boundary types.

The proof has a computer-assisted finite-verification component. After the structural reduction, it remains to verify an explicit finite certificate: a family of nonempty sets of boundary types and root states that is closed under the recursive operations of the decomposition. The same certificate guarantees that the witnesses obtained for completed rooted branches include states in which no boundary defect remains. Thus the crumby-admissible coloring produced by the induction becomes a genuine crumby coloring of the original graph. We state this finite check as the Finite Verification Lemma \ref{lem:verification}. The certificate and a stand-alone checker are supplied with the paper, and the appendix gives additional details about the supplied codes.

\section{Recursive decomposition of rooted outerplanar graphs}\label{sec:decomp}
All graphs are finite and simple. A graph is \emph{subcubic} if every vertex has degree at most three.
\begin{definition}[Crumby coloring]\label{def:crumby}
Let \(G\) be a graph. A \emph{crumby coloring} of \(G\) is a map
\(c:V(G)\to\{\B,\R\}\). Write
\[
G_{\B}=G[c^{-1}(\B)] \qquad\text{and}\qquad G_{\R}=G[c^{-1}(\R)]
\]
for the blue and red induced subgraphs. The coloring \(c\) is \emph{crumby} if
\(
\Delta(G_{\B})\le 1,
\)
while \(G_{\R}\) has no isolated vertices and contains no simple path with three edges; this path is not required to be induced. With respect to the coloring \(c\), an edge is called \emph{blue} if both endpoints are blue, and \emph{red} if both endpoints are red.\footnote{We state the red condition vertex-wise on purpose. The phrase ``every red vertex has a red neighbor'' is unambiguous even when there are no red vertices, whereas the equivalent formulation ``the red induced subgraph has minimum degree at least one'' may be read as undefined for the empty graph. With the vertex-wise formulation, a coloring with no red vertices vacuously satisfies the red conditions; in particular, a single blue vertex is crumby, as needed when isolated vertices are colored blue in the proof of Theorem~\ref{thm:main}.}
\end{definition}
The decomposition works with two kinds of objects, rooted branches and two-terminal intervals, built from two base objects by five constructions. We define the objects first, then the constructions.
\begin{definition}[Rooted branch, two-terminal interval]\label{def:objects}
A \emph{rooted branch} \((G,r)\) is a connected subcubic outerplanar graph \(G\) together with a distinguished vertex \(r\) of degree at most two, called the \emph{root}.

A \emph{two-terminal interval} \((F,p,q)\) is a subcubic outerplanar graph \(F\) together with two ordered distinguished vertices \(p\) and \(q\), the \emph{left} and \emph{right terminal}, joined in \(F\) by a path along the outer boundary, the \emph{boundary path}. 

An interval models what remains of a cyclic block after the deletion of a root lying on it: \(p\) and \(q\) are the two boundary neighbors of the deleted vertex, and at each terminal one incidence is reserved for the edge to that vertex, to be restored only when a new root is attached. Consequently a terminal \(x\) spends one incidence on the boundary path and carries at most one further incidence inside \(F\), its \emph{spare incidence}; we set \(\sigma(x)=1\) if the spare incidence has been used, by a chord or by a bridge to an attached branch, and \(\sigma(x)=0\) if it is still free. Thus \(x\) has degree \(1+\sigma(x)\le 2\) in \(F\). Besides \(\sigma(p)\) and \(\sigma(q)\), we record whether the boundary path has length one or at least two, written \(\ell\in\{1,\ge 2\}\); this flag governs chord addition below, since in a simple graph no chord can join terminals that are already adjacent.
\end{definition}
\begin{definition}[The five constructions]\label{def:constructions}
The \emph{base objects} are the \emph{single rooted vertex}, which is a rooted branch, and the \emph{single-edge interval}: one boundary edge \(pq\) with terminals \(p\) and \(q\), both spare incidences free (\(\sigma(p)=\sigma(q)=0\)) and \(\ell=1\). Larger objects are built by the following constructions. Each is permitted only under the stated conditions; these enforce subcubicity and simplicity at the vertices the construction touches, and no other vertex is affected.
\begin{enumerate}[label=\textnormal{(O\arabic*)},leftmargin=2.7em,itemsep=0.3em]
\item \emph{Rooted bridge composition.} Given one or two rooted branches, add a new vertex \(r\) and join it by a bridge to the root of each. The result is a rooted branch with root \(r\): the new root has degree at most two, and each old root had degree at most two and gains exactly one incidence.
\item \emph{Root restoration.} Given a two-terminal interval \((F,p,q)\), add a new vertex \(x\) adjacent to both \(p\) and \(q\). The result is a rooted branch with root \(x\): the root has degree two, and each terminal receives exactly the incidence reserved for it, so its degree grows from \(1+\sigma\le 2\) to at most three.
\item \emph{Interval concatenation.} Given two intervals \((F_1,p_1,q_1)\) and \((F_2,p_2,q_2)\), identify \(q_1\) with \(p_2\) into a single vertex \(w\); the result is an interval with terminals \(p_1\) and \(q_2\). The identified vertex has degree
\[
(1+\sigma(q_1))+(1+\sigma(p_2))
\]
in the result, so the step is permitted precisely when \(\sigma(q_1)+\sigma(p_2)\le 1\), that is, when not both identified terminals have used their spare incidence. The outer terminals keep their spare-incidence data, and the new boundary path is the concatenation of the two old ones, so \(\ell\ge 2\).
\item \emph{Chord addition.} Given an interval \((F,p,q)\) with \(\sigma(p)=\sigma(q)=0\) and \(\ell \ge 2\), add the edge \(pq\). The first condition is forced by subcubicity, the second by simplicity: terminals at boundary distance one are already adjacent. Afterwards \(\sigma(p)=\sigma(q)=1\).
\item \emph{Terminal branch attachment.} Given an interval \((F,p,q)\), a terminal \(x\in\{p,q\}\) with \(\sigma(x)=0\), and a rooted branch \((G,s)\), join \(s\) to \(x\) by a bridge. The terminal uses its spare incidence, so afterwards \(\sigma(x)=1\); the root \(s\) had degree at most two and gains exactly one incidence.
\end{enumerate}
\end{definition}
Before stating the decomposition result, we recall the internal structure of a cyclic block; the facts about outerplanar graphs used here are standard. A \emph{block} of a graph is a maximal connected subgraph without a cutvertex; every block is a single edge, a \emph{bridge}, or a maximal \(2\)-connected subgraph, which we call a \emph{cyclic block} since it contains a cycle. Every subgraph of an outerplanar graph is outerplanar, so a cyclic block \(B\) has an outerplane embedding, in which the outer face is bounded by a Hamiltonian cycle of \(B\); every edge of \(B\) off this cycle is a \emph{chord}. Now suppose the ambient graph is subcubic. Each vertex of \(B\) spends two incidences on the boundary cycle, so it carries at most one incidence beyond it, used by a chord of \(B\) or by a bridge leaving \(B\), but never both. In particular the chords of \(B\) form a matching, and the outerplane embedding draws them without crossings: the chords of a cyclic block form a noncrossing matching. Figure~\ref{fig:block} shows a rooted cyclic block and the two-terminal interval obtained by deleting its root.
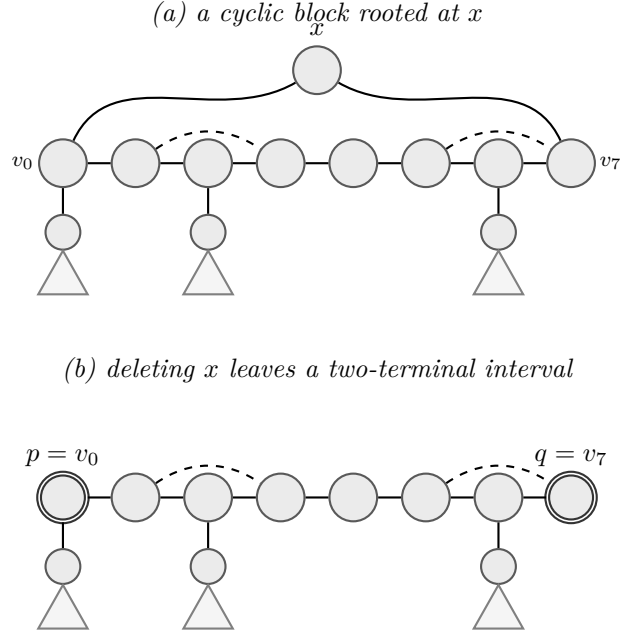
\begin{figure}[t]
\centering
\begin{tikzpicture}[scale=.96,
    branchroot/.style={circle,draw=black!65,fill=black!8,thick,inner sep=1.2pt,minimum size=4.6mm},
    subtree/.style={draw=black!50,fill=black!4,thick},
    term/.style={circle,draw=black!80,double,thick,fill=black!8,inner sep=1.7pt,minimum size=6.3mm}]
  \node[font=\small\itshape] at (0,2.05) {(a) a cyclic block rooted at $x$};
  \node[vgray,label={[font=\small]above:$x$}] (x) at (0,1.28) {};
  \foreach \i/\xx in {0/-3.5,1/-2.5,2/-1.5,3/-0.5,4/0.5,5/1.5,6/2.5,7/3.5} {
    \node[vgray] (v\i) at (\xx,0) {};
  }
  \node[font=\scriptsize] at (-4.05,0) {$v_0$};
  \node[font=\scriptsize] at (4.05,0) {$v_7$};
  \draw[edge] (x) to[out=210,in=65] (v0);
  \draw[edge] (x) to[out=-30,in=115] (v7);
  \foreach \i/\j in {0/1,1/2,2/3,3/4,4/5,5/6,6/7} \draw[edge] (v\i)--(v\j);
  \draw[chord] (v1) to[bend left=35] (v3);
  \draw[chord] (v5) to[bend left=35] (v7);
  \foreach \i in {0,2,6} {
    \node[branchroot] (r\i) at (v\i |- 0,-.95) {};
    \draw[edge] (v\i)--(r\i);
    \draw[subtree] (r\i.south) -- ++(-.34,-.6) -- ++(.68,0) -- cycle;
  }
  \begin{scope}[yshift=-4.6cm]
    \node[font=\small\itshape] at (0,1.75) {(b) deleting $x$ leaves a two-terminal interval};
    \foreach \i/\xx in {0/-3.5,1/-2.5,2/-1.5,3/-0.5,4/0.5,5/1.5,6/2.5,7/3.5} {
      \node[vgray] (w\i) at (\xx,0) {};
    }
    \node[term] at (w0) {};
    \node[term] at (w7) {};
    \node[font=\small] at (-3.5,.55) {$p=v_0$};
    \node[font=\small] at (3.5,.55) {$q=v_7$};
    \foreach \i/\j in {0/1,1/2,2/3,3/4,4/5,5/6,6/7} \draw[edge] (w\i)--(w\j);
    \draw[chord] (w1) to[bend left=35] (w3);
    \draw[chord] (w5) to[bend left=35] (w7);
    \foreach \i in {0,2,6} {
      \node[branchroot] (s\i) at (w\i |- 0,-.95) {};
      \draw[edge] (w\i)--(s\i);
      \draw[subtree] (s\i.south) -- ++(-.34,-.6) -- ++(.68,0) -- cycle;
    }
  \end{scope}
\end{tikzpicture}
\caption{(a) A cyclic block rooted at $x$. The root uses its two cycle edges, so every other boundary vertex has at most one spare incidence, used by one chord of the noncrossing matching or by one bridge to a rooted child branch (triangles). (b) Deleting $x$ opens the block into a two-terminal interval with terminals $p=v_0$ and $q=v_7$.}
\label{fig:block}
\end{figure}
\begin{proposition}[Structural decomposition]\label{prop:decomp}
Let \(H\) be a connected subcubic outerplanar graph, rooted at a vertex \(r\) of degree at most two. Then \((H,r)\) is generated from the base objects by finitely many applications of the constructions \textnormal{(O1)--(O5)}. More explicitly, every generated rooted branch is a single rooted vertex, or arises by rooted bridge composition \textnormal{(O1)} from one or two smaller rooted branches, or by root restoration \textnormal{(O2)} from a two-terminal interval; and every generated two-terminal interval arises from single-edge intervals by interval concatenation \textnormal{(O3)}, chord addition \textnormal{(O4)}, and terminal branch attachment \textnormal{(O5)}, subject to the conditions of Definition~\ref{def:constructions}.
\end{proposition}
\begin{proof}
We argue by strong induction on the number of vertices: we show that every rooted branch \((H,r)\) is generated in the stated explicit form, assuming that every rooted branch with fewer vertices is. Inside the induction step a second, nested induction generates the required two-terminal intervals; its measure is the boundary length, and it is set up below where the intervals first appear. Every rooted branch to which the outer hypothesis is applied will be a proper subgraph of \(H\), hence smaller.

If \(H\) has one vertex, it is the single rooted vertex, a base object. Otherwise we distinguish two cases according to whether the root lies in a cyclic block.

\emph{Case 1: \(r\) lies in no cyclic block.} Then every edge incident with \(r\) is a bridge. For each such bridge \(rs\), let \(H_s\) be the component of \(H-rs\) containing \(s\), rooted at \(s\). Deleting the bridge removes one incidence of \(s\), so \(d_{H_s}(s)\le 2\) and \((H_s,s)\) is a rooted branch; it is a proper subgraph of \(H\), as it misses \(r\), so it is generated by the outer induction hypothesis. Since \(d_H(r)\le 2\), there are at most two bridges at \(r\), and applying rooted bridge composition (O1) to the branches \((H_s,s)\) recovers \((H,r)\).

\emph{Case 2: \(r\) lies in a cyclic block \(B\).} Since \(B\) is \(2\)-connected, \(r\) has at least two incidences in \(B\); since \(d_H(r)\le 2\), these are all incidences of \(r\) in \(H\), and they are the two boundary-cycle edges of \(B\) at \(r\). Hence \(r\) meets no bridge, no other block, and no chord of \(B\). Fix an outerplane embedding of \(B\) and write its outer boundary cycle as
\[
r,v_0,v_1,\ldots,v_m,r .
\]
Every boundary vertex \(v_i\) spends two incidences on this cycle and carries at most one spare incidence, used by a chord of \(B\) or by a bridge to a rooted branch hanging from \(v_i\), never both. Let \(F=H-r\), with terminals \(v_0\) and \(v_m\). Then \((F,v_0,v_m)\) is a two-terminal interval: its boundary path is \(v_0v_1\cdots v_m\), each terminal has lost its edge to \(r\) and so has degree at most two in \(F\), and the spare incidences of the terminals realize the data \(\sigma\). If \(F\) is generated, then root restoration (O2) applied to \(F\) reconstructs \((H,r)\). It therefore remains to generate \(F\).

Here the nested induction enters. For \(i<j\), call the consecutive set \(v_i,\ldots,v_j\) \emph{internally chord-closed} if no chord of \(B\) has one endpoint among the interior vertices \(v_{i+1},\ldots,v_{j-1}\) and the other endpoint outside \(v_i,\ldots,v_j\). Chords incident with one of the two terminals are allowed to leave the set; in the smaller interval they are simply pending external incidences and are not present. Let \(I[i,j]\) be the interval consisting of the boundary path \(v_i\cdots v_j\), all chords of \(B\) with both ends among \(v_i,\ldots,v_j\), and all rooted branches hanging from vertices among \(v_i,\ldots,v_j\), with terminals \(v_i\) and \(v_j\). If a terminal has a chord to a vertex outside the set, that chord is omitted from \(I[i,j]\) and the corresponding spare incidence is free there. The definition of internal chord-closedness ensures that no nonterminal vertex of \(I[i,j]\) has such an omitted chord. We prove, by induction on the boundary length \(j-i\), that \(I[i,j]\) is generated for every internally chord-closed pair \((i,j)\). The desired interval \(F\) is the instance \(i=0\), \(j=m\), and the full vertex set is internally chord-closed. Throughout, every hanging branch is rooted at the end vertex of its bridge, where it has root degree at most two, and it is a proper subgraph of \(H\); it is therefore generated by the outer induction hypothesis and may be attached by (O5) whenever the relevant spare incidence is free.

\emph{Base case: \(j-i=1\).} No chord has both ends in \(\{v_i,v_{i+1}\}\), since these vertices are already adjacent along the boundary. The interval \(I[i,i+1]\) is thus built from the single-edge interval \(v_iv_{i+1}\) by attaching the hanging branches at \(v_i\) and at \(v_{i+1}\), when present, by terminal branch attachment (O5). Each such attachment is permitted because the corresponding terminal has no chord using its spare incidence; if a chord leaves the terminal to the outside, then no hanging branch is present there and that pending chord is deliberately omitted.

\emph{Inductive step: \(j-i\ge 2\).} Consider the left terminal \(v_i\).

Suppose first that no chord joins \(v_i\) to a vertex of \(v_{i+1},\ldots,v_j\). Start from the single-edge interval \(v_iv_{i+1}\) and attach the hanging branch at \(v_i\), if present, by (O5). The pair \((i+1,j)\) is internally chord-closed: a chord from an interior vertex of \(v_{i+1},\ldots,v_j\) to \(v_i\) would contradict the present assumption, and a chord leaving \(v_i,\ldots,v_j\) would contradict internal chord-closedness of \((i,j)\). Hence \(I[i+1,j]\) has smaller boundary length and is generated by the inner hypothesis. Concatenating (O3) the single-edge piece with it yields \(I[i,j]\); the concatenation is permitted because the copy of \(v_{i+1}\) in the single-edge piece has a free spare incidence.

Suppose now that a chord \(v_iv_k\) joins \(v_i\) to a vertex of \(v_{i+1},\ldots,v_j\). Then \(i+2\le k\le j\) by simplicity. Since a vertex's single spare incidence is exhausted by a chord, neither \(v_i\) nor \(v_k\) carries a hanging branch or a second chord. The pairs \((i+1,k)\) and, when \(k<j\), \((k,j)\) are internally chord-closed. Indeed, a chord from an interior vertex of either smaller pair to a vertex outside that smaller pair either leaves the original interval \(v_i,\ldots,v_j\), is incident with \(v_i\) or \(v_k\) contrary to the matching property of the chords, or crosses the chord \(v_iv_k\) in the fixed outerplane embedding.

By the inner hypothesis, \(I[i+1,k]\) is generated. Start from the single-edge interval \(v_iv_{i+1}\) and concatenate (O3) it with \(I[i+1,k]\); this is permitted because the copy of \(v_{i+1}\) in the single-edge piece has a free spare incidence. The result is exactly \(I[i,k]\) with the chord \(v_iv_k\) omitted: nothing else is missing at \(v_i\), which carries no hanging branch and no other chord. Both terminal spare incidences of this interval are free, since the only spare use at \(v_i\) and at \(v_k\) is the chord itself, not yet present, and its boundary length \(k-i\) is at least two. Chord addition (O4) therefore inserts \(v_iv_k\), producing \(I[i,k]\). If \(k=j\), this is the required interval \(I[i,j]\). If \(k<j\), then \(I[k,j]\) is generated by the inner hypothesis; its left terminal \(v_k\) has the pending chord \(v_iv_k\) omitted and hence has free spare incidence. Concatenating (O3) \(I[i,k]\) with \(I[k,j]\) is therefore permitted, since the used spare incidence of \(v_k\) in the left piece is matched with a free one in the right piece, and the result is \(I[i,j]\).

Both inductions are well founded: each appeal to the inner hypothesis strictly decreases the boundary length, and each appeal to the outer hypothesis concerns a rooted branch with fewer vertices. Hence \(F\) is generated, and with it \((H,r)\), in the explicit form stated.
\end{proof}

In particular, since every nonempty outerplanar graph has a vertex of degree at most two, Proposition~\ref{prop:decomp} implies that, after choosing such a vertex as the root, every connected subcubic outerplanar graph is generated by the rooted constructions above.

\section{Boundary information for partial colorings}\label{sec:boundary}
We now describe the coloring information that must be retained at the terminals of an interval, and prove that it is sufficient. We begin with the possible shapes of a red component, used repeatedly below.
\begin{lemma}[Red components]\label{lem:redcomp}
In any crumby coloring of a subcubic graph, every red component is one of
\[
K_2,\qquad P_3,\qquad K_{1,3},\qquad K_3.
\]
Conversely, any coloring whose blue induced subgraph has maximum degree at most one, and each of whose red components is one of these four graphs, is crumby.
\end{lemma}
\begin{proof}
Let \(C\) be a red component. Then \(C\) is connected, has at least two vertices, since a one-vertex component would be an isolated red vertex, has maximum degree at most three by subcubicity, and contains no simple path with three edges; as in Definition~\ref{def:crumby}, the forbidden path need not be induced.

Suppose first that \(C\) is a tree. Two vertices at distance three are joined by a path with three edges, so \(C\) has diameter at most two, and a tree of diameter at most two is a star \(K_{1,t}\) with \(t\ge 1\). The degree bound at the center gives \(t\le 3\), so \(C\) is \(K_2\), \(P_3\), or \(K_{1,3}\).

Suppose now that \(C\) contains a cycle. A cycle on four or more vertices contains a simple path with three edges, so every cycle of \(C\) is a triangle; fix one, on the vertices \(x,y,z\). If \(C\) had any further vertex, then, since \(C\) is connected, some edge of \(C\) would join a vertex \(u\) outside the triangle to a vertex of the triangle, say to \(x\); but then \(u\,x\,y\,z\) would be a simple red path with three edges. Hence \(V(C)=\{x,y,z\}\), and \(C\) is the triangle \(K_3\).

Conversely, suppose that the blue induced subgraph has maximum degree at most one and that every red component is one of the four listed graphs. Each of these has at least two vertices, so no red vertex is isolated; and none of them contains a simple path with three edges, so neither does the red induced subgraph, since a path is confined to a single component. The coloring is therefore crumby.
\end{proof}
When an interval or branch is used inside a larger graph, a terminal or root may still acquire one further neighbor. We therefore allow one temporary defect: an exceptional vertex may be a red vertex with no red neighbor.
\begin{definition}[Crumby-admissible colorings]\label{def:admissible}
Let the \emph{exceptional vertices} of a rooted branch \((G,r)\) be its root, and those of a two-terminal interval \((F,p,q)\) be its two terminals. A coloring of a rooted branch or of a two-terminal interval is \emph{crumby-admissible} if the blue induced subgraph has maximum degree at most one, the red induced subgraph contains no simple path with three edges, and every red vertex, except possibly the exceptional ones, has a red neighbor. In a rooted branch, the status of the root is one of the five \emph{root states}:
\begin{center}\small
\begin{tabular}{@{}cl@{}}
\toprule
root state & meaning at the root \(r\)\\
\midrule
\(B_0\) & blue, with no blue neighbor\\
\(B_1\) & blue, with one blue neighbor\\
\(R_0\) & red, with no red neighbor (the temporary defect)\\
\(R_1\) & red, longest red path from \(r\) has length \(1\)\\
\(R_2\) & red, longest red path from \(r\) has length \(2\)\\
\bottomrule
\end{tabular}
\end{center}
\end{definition}
When a branch is completed, its root acquires no further neighbor, so the defect \(R_0\) is no longer allowed: a crumby-admissible coloring of a completed branch is crumby precisely when its root state lies in the \emph{final-legal} set
\begin{equation}\label{eq:finallegal}
\FL=\{B_0,\ B_1,\ R_1,\ R_2\}.
\end{equation}
For a red vertex \(v\), write \(\depth(v)\in\{0,1,2\}\) for the maximum number of edges on a simple red path starting at \(v\); crumby-admissibility forbids depth three or more, and \(\depth(v)=0\) marks an exceptional red vertex with no red neighbor.

Lemma~\ref{lem:redcomp} extends to crumby-admissible colorings in the following form: every red component with at least two vertices is one of \(K_2\), \(P_3\), \(K_{1,3}\), \(K_3\), while a one-vertex red component consists of an exceptional vertex. We use this extension repeatedly below.

The single rule for joining a branch to the rest of the graph across a bridge is summarized in Table~\ref{tab:bridge} and established in Lemma~\ref{lem:bridgestate} below. When a rooted branch with root \(r\) is attached to a vertex \(v\) by the bridge \(vr\), the table lists the legal child states and the contribution made at \(v\); the contribution must then be compatible with what is already present at \(v\), namely the blue degree at \(v\) must stay at most one and the longest red path through \(v\) must have at most two edges. Because the arms contributed at \(v\) lie in distinct child branches and are therefore internally disjoint, this last condition is exactly that its two largest such arms sum to at most two. The vertex \(v\) is permitted to be red-isolated only while it is itself still an exceptional root or terminal.
\begin{table}[ht]
\centering
\caption{Bridge-attachment rule. A rooted branch of the given root state is attached to a vertex \(v\) by one bridge.}
\label{tab:bridge}
\small
\begin{tabular}{@{}cll@{}}
\toprule
color of \(v\) & legal child state & contribution at \(v\)\\
\midrule
blue & \(B_0\) & one blue neighbor at \(v\)\\
blue & \(R_1,\ R_2\) & none (\(B_1,R_0\) forbidden)\\
red  & \(B_0,\ B_1\) & none\\
red  & \(R_0\) & a red arm of length \(1\) at \(v\)\\
red  & \(R_1\) & a red arm of length \(2\) at \(v\) (\(R_2\) forbidden)\\
\bottomrule
\end{tabular}
\end{table}
\begin{lemma}[Bridge attachment is state-determined]\label{lem:bridgestate}
Let a rooted branch \(H\) with root \(r\), carrying a crumby-admissible coloring with root state \(s\), be attached to a vertex \(v\) by the bridge \(vr\). Then the constraints that crumby-admissibility imposes at \(r\) and along the bridge, given the color of \(v\), and the contribution recorded at \(v\) in Table~\ref{tab:bridge} depend only on \(s\), not on the coloring of \(H\) realizing it. 
\end{lemma}
\begin{proof}
The bridge is the only edge between \(H\) and the rest of the graph, and no vertex changes color, so crumby-admissibility can fail, and same-color data can change, only at \(v\) and \(r\). 

Two observations reduce everything at \(r\) to \(s\).  First, the blue datum at \(r\) is its blue degree, recorded by \(s\in\{B_0,B_1\}\). Second, a simple red path that enters \(H\) does so along \(vr\) and continues from \(r\) inside \(H\); it cannot leave \(H\) again and cannot revisit \(r\), so its length is at most \(1+\depth(r)\), with equality attained along a longest simple red path from \(r\), and \(\depth(r)\) is recorded by \(s\in\{R_0,R_1,R_2\}\).

Now take the four color pairs. If \(v\) and \(r\) are both blue, \(vr\) is a blue edge and raises both blue degrees by one: legality at \(r\) is exactly \(s=B_0\), and the contribution at \(v\) is one blue neighbor. If the colors differ, \(vr\) is neither blue nor red and changes no same-color datum; the only constraint is that a red \(r\) attached to a blue \(v\) must already have a red neighbor inside \(H\), which excludes exactly \(s=R_0\), while \(s\in\{R_1,R_2\}\) and \(s\in\{B_0,B_1\}\) contribute nothing at \(v\). If \(v\) and \(r\) are both red, the bridge contributes at \(v\) one red arm of length exactly \(1+\depth(r)\) by the second observation: \(s=R_0\), whose root thereby gains its red neighbor, gives an arm of length \(1\); \(s=R_1\) gives an arm of length \(2\); and for \(s=R_2\) the arm is itself a red path with three edges, so the attachment is illegal. Every quantity used is a function of \(s\) alone, which yields exactly the entries of Table~\ref{tab:bridge}.
\end{proof}
\begin{definition}[Boundary types]\label{def:boundary}
The \emph{boundary type} of a crumby-admissible coloring of an interval \((F,p,q)\) records the colors of \(p\) and \(q\) and the local same-color data at the terminals. There are \(31\) types, in four families:
\begin{center}\small
\begin{tabularx}{\textwidth}{@{}l l X r@{}}
\toprule
family & symbol & data & count\\
\midrule
both blue & \(BB_{ij}\), \(i,j\in\{0,1\}\) & blue degrees \(i\) at \(p\), \(j\) at \(q\) & 4\\
one of each & \(BR_{i\alpha}\), \(RB_{\alpha j}\) & blue degree and red depth, with \(i,j\in\{0,1\}\) and \(\alpha\in\{0,1,2\}\) & 12\\
red--red, two components & \(RR^d_{\alpha\beta}\), \(\alpha,\beta\in\{0,1,2\}\) & depths \(\alpha\) at \(p\), \(\beta\) at \(q\) & 9\\
red--red, one component & \(RR^c_\tau\), \(\tau\in\{E,C_p,C_q,P,S,T\}\) & ordered shape of the component & 6\\
\bottomrule
\end{tabularx}
\end{center}
The six ordered connected shapes are: the red edge \(pq\) (\(E\)); a red \(P_3\) with center \(p\), \(q\) a leaf (\(C_p\)), and its mirror (\(C_q\)); a red \(P_3\) with \(p,q\) the two leaves (\(P\)); a red \(K_{1,3}\) with \(p,q\) two of its leaves (\(S\)); and a red triangle through \(p,q\) (\(T\)).
\end{definition}
Every crumby-admissible coloring of a two-terminal interval has exactly one boundary type. Indeed, a blue terminal has blue degree \(0\) or \(1\). A red terminal whose red component avoids the other terminal has depth \(0\), \(1\), or \(2\). Finally, if the two terminals share one red component, that component has at least two vertices, so by the extension of Lemma~\ref{lem:redcomp} above it is one of \(K_2\), \(P_3\), \(K_{1,3}\), \(K_3\); since each terminal has degree at most two in the interval (Definition~\ref{def:objects}), neither terminal can be the center of the \(K_{1,3}\), and the ordered shapes that remain are exactly the six listed. The six connected shapes are drawn in Figure~\ref{fig:redshapes}.
\begin{figure}[ht]
\centering
\begin{tikzpicture}[scale=.93, sub/.style={font=\scriptsize,black!70}]
  \node at (0,1.5) {$E$};
  \node[vred] (e1) at (-.45,0) {\scriptsize$p$};
  \node[vred] (e2) at (.45,0) {\scriptsize$q$};
  \draw[edge,RedV] (e1)--(e2);
  \node[sub] at (0,-1.25) {red edge $pq$};
  \node at (2.6,1.5) {$C_p$};
  \node[vred] (cp0) at (2.2,0) {\scriptsize$p$};
  \node[vred] (cp1) at (3.0,.55) {\scriptsize$q$};
  \node[vred] (cp2) at (3.0,-.55) {};
  \draw[edge,RedV] (cp0)--(cp1);
  \draw[edge,RedV] (cp0)--(cp2);
  \node[sub] at (2.6,-1.25) {center $p$};
  \node at (5.0,1.5) {$C_q$};
  \node[vred] (cq0) at (5.4,0) {\scriptsize$q$};
  \node[vred] (cq1) at (4.6,.55) {\scriptsize$p$};
  \node[vred] (cq2) at (4.6,-.55) {};
  \draw[edge,RedV] (cq0)--(cq1);
  \draw[edge,RedV] (cq0)--(cq2);
  \node[sub] at (5.0,-1.25) {center $q$};
  \node at (7.65,1.5) {$P$};
  \node[vred] (p0) at (6.9,0) {\scriptsize$p$};
  \node[vred] (p1) at (7.65,0) {};
  \node[vred] (p2) at (8.4,0) {\scriptsize$q$};
  \draw[edge,RedV] (p0)--(p1)--(p2);
  \node[sub] at (7.65,-1.25) {leaves of $P_3$};
  \node at (9.95,1.5) {$S$};
  \node[vred] (s0) at (9.95,0) {};
  \node[vred] (s1) at (9.25,-.55) {\scriptsize$p$};
  \node[vred] (s2) at (10.65,-.55) {\scriptsize$q$};
  \node[vred] (s3) at (9.95,.85) {};
  \draw[edge,RedV] (s0)--(s1);
  \draw[edge,RedV] (s0)--(s2);
  \draw[edge,RedV] (s0)--(s3);
  \node[sub] at (9.95,-1.25) {leaves of $K_{1,3}$};
  \node at (12.3,1.5) {$T$};
  \node[vred] (t1) at (11.85,-.35) {\scriptsize$p$};
  \node[vred] (t2) at (12.75,-.35) {\scriptsize$q$};
  \node[vred] (t3) at (12.3,.45) {};
  \draw[edge,RedV] (t1)--(t2)--(t3)--(t1);
  \node[sub] at (12.3,-1.25) {triangle};
\end{tikzpicture}
\caption{The six ordered connected red two-terminal shapes of Definition~\ref{def:boundary}: both terminals are red and lie in one red component.}
\label{fig:redshapes}
\end{figure}
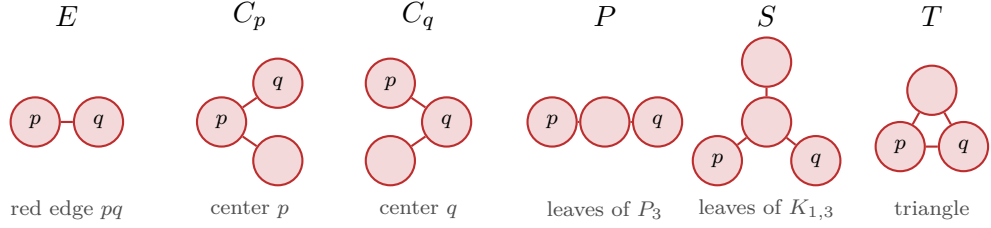
Two separate red depths do not always suffice when the terminals share a red component. Types \(P\) and \(S\) both present a red path of length two leaving each terminal, so they have identical terminal depths; yet adding the terminal edge \(pq\) turns \(P\) into a triangle, which is still crumby, while in \(S\) it creates a red path with three edges, which is forbidden. The two depths alone do not determine the outcome, so the boundary type must record the full ordered shape; this is why the six connected shapes are distinguished.

Table~\ref{tab:red-arms} records the data of the six connected shapes used repeatedly below: the arms at the two terminals, whether \(pq\) is an edge of the shape, and which spare incidences the shape forces to be used.
\begin{table}[ht]
\centering
\caption{Data of the six connected boundary types. An entry \(\{1,1\}\) means two arms of length one. A terminal of red degree two has spent its spare incidence, so its flag is forced to \emph{used}.}
\label{tab:red-arms}
\small
\begin{tabular}{@{}c c c c c@{}}
\toprule
Type & Arms at \(p\) & Arms at \(q\) & Edge \(pq\) & Forced used flags\\
\midrule
\(E\)   & \(\{1\}\)   & \(\{1\}\)   & yes & none\\
\(C_p\) & \(\{1,1\}\) & \(\{2\}\)   & yes & at \(p\)\\
\(C_q\) & \(\{2\}\)   & \(\{1,1\}\) & yes & at \(q\)\\
\(P\)   & \(\{2\}\)   & \(\{2\}\)   & no  & none\\
\(S\)   & \(\{2\}\)   & \(\{2\}\)   & no  & none\\
\(T\)   & \(\{2,2\}\) & \(\{2,2\}\) & yes & at \(p\) and \(q\)\\
\bottomrule
\end{tabular}
\end{table}

We can now state the central lemma. Besides the boundary type, it carries the three structural flags of Definition~\ref{def:objects}: whether the spare incidence at each terminal is still free, written \(e_L,e_R\) (so \(e_L=\text{free}\) means \(\sigma(p)=0\), and likewise at \(q\)), and whether the boundary path between the terminals has length one or at least two, written \(\ell\).
\begin{lemma}[Boundary sufficiency]\label{lem:sufficiency}
Let \((F,p,q)\) be a two-terminal interval carrying a crumby-admissible coloring of boundary type \(\pi\), with structural flags \((e_L,e_R,\ell)\).
\begin{enumerate}[label=\textnormal{(\alph*)},leftmargin=2.1em,itemsep=0.25em]
\item For the unary constructions --- chord addition \textnormal{(O4)}, attachment of a rooted branch of a given root state at a given terminal \textnormal{(O5)}, and root restoration \textnormal{(O2)} --- the type \(\pi\) and the flags determine whether the construction is legal and, when it is, the boundary type or root state of the result.
\item For interval concatenation \textnormal{(O3)}, the two boundary types and the two flag triples determine whether the construction is legal and, when it is, the boundary type and flags of the result.
\end{enumerate}
In particular, two crumby-admissible colorings with the same boundary type and flags are interchangeable under every later construction, whether it is applied to the interval alone or composes it with a second interval.
\end{lemma}
\begin{proof}
Each construction adds edges only at the terminals of its arguments, so it suffices to show, construction by construction, that legality and the boundary data of the result are functions of the recorded types and flags, independently of the colorings realizing them. The structural flags of the result transform mechanically, as recorded in Definition~\ref{def:constructions}, so only the color data need argument. For brevity we call terminal branch attachment (O5) \emph{decoration}.

For a red vertex \(t\) and a red edge \(tu\), the \emph{arm} of \(tu\) at \(t\) is the maximum number of edges of a simple red path starting at \(t\) with first edge \(tu\); thus \(\depth(t)\) is the largest arm at \(t\), or \(0\) if \(t\) has no red neighbor. Call a red terminal \emph{separated} if its red component avoids the other terminal. Two claims show that at a separated terminal the single number \(\depth(t)\) carries all the information the constructions need; at the terminals of a connected type \(RR^c_\tau\), the arms are instead prescribed by the shape \(\tau\), as listed in Table~\ref{tab:red-arms}.

\emph{Claim 1: if \(\sigma(t)=0\), then \(t\) has at most one arm, of length \(\depth(t)\).} Indeed, \(t\) then has degree one in \(F\) (Definition~\ref{def:objects}), so it meets at most one red edge.

\emph{Claim 2: a separated red terminal \(t\) with two arms has \(\sigma(t)=1\), and every construction that reaches \(t\) sees only \(\depth(t)\).} Two arms give \(t\) red degree two, so both incidences of \(t\) are spent and \(\sigma(t)=1\). Then \(t\) is a degree-two vertex of its red component, which has at least two vertices and is therefore one of \(K_2\), \(P_3\), \(K_{1,3}\), \(K_3\) by the extension of Lemma~\ref{lem:redcomp}; of these, only \(P_3\) and \(K_3\) have a vertex of degree two. So either \(t\) is the center of a red \(P_3\), with arms \(\{1,1\}\) and \(\depth(t)=1\), or \(t\) lies on a red triangle, with arms \(\{2,2\}\) and \(\depth(t)=2\); the triangle contains no red three-edge path, its two arms at \(t\) overlapping instead of concatenating. Since \(\sigma(t)=1\), chord addition and decoration do not apply at \(t\), and the two constructions that still reach \(t\), concatenation and root restoration, each attach exactly one new edge at \(t\). A new simple red path beginning with that edge enters the red component of \(t\) and cannot return to \(t\), so within that component its length is exactly \(1+\depth(t)\), whichever arm it follows. Hence legality and outcome depend on \(\depth(t)\) alone. This proves the claim, and we take the constructions in turn, deriving the complete list of rules for each.

\smallskip\noindent\emph{Concatenation (O3).}
Let \(F_1\) (terminals \(p_1,q_1\), type \(\pi_1\)) and \(F_2\) (terminals \(p_2,q_2\), type \(\pi_2\)) be concatenated by identifying \(q_1\) with \(p_2\) into one vertex \(w\); the surviving terminals are \(p_1\) and \(q_2\), and (O3) requires the identified terminals to have the same color and \(\sigma(q_1)+\sigma(p_2)\le 1\). The identification creates no new edge: \(w\) simply inherits the edges of \(q_1\) and of \(p_2\). Hence every newly created same-color path passes through \(w\), and only \(p_1\), \(w\), and \(q_2\) can change status. Moreover \(F_1\) and \(F_2\) meet only in \(w\), so the red components of \(p_1\) and of \(q_2\) can never merge with each other except through the component of \(w\).

\emph{Blue identification.} The only same-color datum that changes is the blue degree of \(w\), which becomes the sum of the blue degrees recorded at \(q_1\) and \(p_2\); the step is legal exactly when this sum is at most one. The data at \(p_1\) and \(q_2\) are untouched, and no red edge is created anywhere, so the resulting type consists of the datum of \(\pi_1\) at \(p_1\) and the datum of \(\pi_2\) at \(q_2\) (concatenation is denoted by $*$):
\begin{itemize}[leftmargin=1.4em,itemsep=0.1em]
\item \(BB_{i,j_1}\ast BB_{i_2,j}\mapsto BB_{ij}\), \quad \(BB_{i,j_1}\ast BR_{i_2,\beta}\mapsto BR_{i\beta}\),
\item \(RB_{\alpha,j_1}\ast BB_{i_2,j}\mapsto RB_{\alpha j}\), \quad \(RB_{\alpha,j_1}\ast BR_{i_2,\beta}\mapsto RR^d_{\alpha\beta}\),
\end{itemize}
in each case legal exactly when \(j_1+i_2\le 1\). A surviving red terminal stays separated: its component lies in one \(F_i\) and avoids the blue vertex \(w\).

\emph{Red identification.} Write \(a\) and \(b\) for the largest arms at \(q_1\) in \(F_1\) and at \(p_2\) in \(F_2\): for a separated terminal this is the recorded depth (Claims~1 and~2), for a connected type it is read from Table~\ref{tab:red-arms}. Two conditions decide legality. First, \(w\) is now interior, so it must have a red neighbor: this fails exactly when \(a=b=0\). Second, a new simple red path through \(w\) is an arm in \(F_1\) followed by an arm in \(F_2\); these are internally disjoint, so the longest new path has length \(a+b\), and the step is legal exactly when \(a+b\le 2\). Paths avoiding \(w\) are old, so no other violation can arise. For the resulting type, the red components of \(q_1\) and \(p_2\) merge into the component of \(w\); the surviving terminal \(p_1\) lies in it if and only if \(\pi_1\) is a connected type, and symmetrically for \(q_2\). Three cases arise.
\begin{itemize}[leftmargin=1.4em,itemsep=0.1em]
\item \emph{Neither \(\pi_1\) nor \(\pi_2\) is connected.} The components of \(p_1\) and \(q_2\) are untouched, so both terminals keep their data, and the result consists of the datum of \(\pi_1\) at \(p_1\) and of \(\pi_2\) at \(q_2\): \(BR_{i,a}\ast RB_{b,j}\mapsto BB_{ij}\), \ \(BR_{i,a}\ast RR^d_{b\beta}\mapsto BR_{i\beta}\), \ \(RR^d_{\alpha a}\ast RB_{b,j}\mapsto RB_{\alpha j}\), \ \(RR^d_{\alpha a}\ast RR^d_{b\beta}\mapsto RR^d_{\alpha\beta}\), each legal under the two conditions above.
\item \emph{Exactly one of them is connected, say \(\pi_1\) (the other case is its mirror).} Then \(p_1\) joins the component of \(w\), while \(q_2\) keeps its datum and stays separated from \(p_1\). A longest red path from \(p_1\) either stays inside the old shape, giving its depth \(d\) there, or reaches \(w\) and continues along the \(F_2\)-side arm, giving \(\lambda+b\), where \(\lambda\) is the length of a longest simple red path from \(p_1\) to \(w\) inside the shape; the new depth of \(p_1\) is therefore \(\max\{d,\,\lambda+b\}\). With the values of Table~\ref{tab:red-arms} and the legality bound \(a+b\le 2\), this evaluates to: \(1+b\) for \(\pi_1=E\) (\(d=\lambda=1\)); \(1\) for \(\pi_1=C_p\) (\(d=\lambda=1\), and \(a=2\) forces \(b=0\)); and \(2\) for \(\pi_1\in\{C_q,P,S,T\}\). Indeed, for \(C_q\) we have \(d=2\) and \(\lambda=1\), so \(\max\{2,1+b\}=2\); for \(P\), \(S\), and \(T\) we have \(d=2\) and \(\lambda=2\), and there \(a=2\) forces \(b=0\), so \(\lambda+b=2\) as well. The result is \(RR^d_{\delta\beta}\) or \(RB_{\delta j}\) with \(\delta\) this new depth.
\item \emph{Both are connected.} Now \(a,b\ge 1\), so legality forces \(a=b=1\), which by Table~\ref{tab:red-arms} means \(\pi_1\in\{E,C_q\}\) and \(\pi_2\in\{E,C_p\}\); every other pair creates a red path of length \(a+b\ge 3\). The pair \((C_q,C_p)\) is structurally unavailable: both identified terminals have red degree two, so \(\sigma(q_1)=\sigma(p_2)=1\), violating (O3). The three legal pairs merge the two shapes at \(w\): \(E\ast E\) gives the red path \(p_1\,w\,q_2\), type \(P\); \(E\ast C_p\) gives the red star with center \(w\) and leaves \(p_1\), \(q_2\), and the old leaf, type \(S\); \(C_q\ast E\) gives the same star, type \(S\). In each case the longest path through \(w\) is \(1+1=2\), confirming legality.
\end{itemize}
Every quantity used is part of \(\pi_1\) and \(\pi_2\), which proves (b).

\smallskip\noindent\emph{Chord addition (O4).}
The construction requires \(\sigma(p)=\sigma(q)=0\) and \(\ell \ge 2\), so by Claim~1 each terminal carries at most one arm, of length its depth, and \(pq\) is a genuinely new edge, which changes same-color data only at \(p\) and \(q\). The complete rules:
\begin{itemize}[leftmargin=1.4em,itemsep=0.1em]
\item \(BB_{00}\mapsto BB_{11}\): the chord is a blue edge and raises both blue degrees by one; every other \(BB_{ij}\) is forbidden, since a terminal would reach blue degree two.
\item \(BR_{i\alpha}\mapsto BR_{i\alpha}\) and \(RB_{\alpha j}\mapsto RB_{\alpha j}\), always legal: the chord is bichromatic and changes no same-color datum.
\item \(RR^d_{\alpha\beta}\): the chord joins the two components and creates a red path of length \(\alpha+1+\beta\) through \(pq\), so legality requires \(\alpha+\beta\le 1\). This leaves \(RR^d_{00}\mapsto RR^c_E\), \(RR^d_{10}\mapsto RR^c_{C_p}\), and \(RR^d_{01}\mapsto RR^c_{C_q}\): the merged component is the edge \(pq\), respectively a \(P_3\) centered at the terminal of depth one.
\item \(RR^c_\tau\) with \(\tau\in\{E,C_p,C_q,T\}\): unavailable. These shapes contain the edge \(pq\) (Table~\ref{tab:red-arms}), and an existing edge \(pq\) forces \(\ell=1\) or both spare incidences used, so no realization of these types meets the availability conditions.
\item \(RR^c_P\mapsto RR^c_T\): the chord closes the path \(p\,c\,q\) into a red triangle, whose longest simple path has two edges. \(RR^c_S\) is forbidden: with \(c\) the center and \(w\) the third leaf of the star, \(w\,c\,p\,q\) would be a red three-edge path.
\end{itemize}

\smallskip\noindent\emph{Decoration (O5).}
Attaching a rooted branch of root state \(s\) at a terminal \(t\) requires \(\sigma(t)=0\), so by Claim~1 the red structure at \(t\) is a single arm of length \(\depth(t)\). By Lemma~\ref{lem:bridgestate}, the bridge contributes at \(t\) precisely what Table~\ref{tab:bridge} prescribes for a child in state \(s\), independently of the coloring realizing \(s\); the incumbent arm and the contributed arm are internally disjoint, the latter lying in the attached branch, so the longest red path through \(t\) is their sum. The other terminal is untouched. The complete rules:
\begin{itemize}[leftmargin=1.4em,itemsep=0.1em]
\item \(t\) blue with blue degree \(i\): \(s=B_0\) is legal exactly when \(i=0\), raising the blue degree to one. \(s=B_1\) is forbidden: the child root would gain a second blue neighbor. \(s\in\{R_1,R_2\}\) is legal and changes nothing, the bridge being bichromatic. \(s=R_0\) is forbidden: the child root, no longer exceptional, would stay red with no red neighbor.
\item \(t\) red and separated, of depth \(\alpha\): \(s\in\{B_0,B_1\}\) is legal and changes nothing. \(s=R_0\) contributes an arm of length one: legal exactly when \(\alpha+1\le 2\), that is \(\alpha\le 1\), and the new depth is \(1\); for \(\alpha=0\) both \(t\) and the child root thereby gain their red neighbor. \(s=R_1\) contributes an arm of length two: legal exactly when \(\alpha=0\), and the new depth is \(2\). \(s=R_2\) is forbidden: its arm is a red three-edge path. The terminal stays separated, so the family (\(BR\), \(RB\), or \(RR^d\)) is preserved and only the depth is updated.
\item \(t\) red in a connected type: \(s\in\{B_0,B_1\}\) is legal and changes nothing. A red child contributes an arm of length at least one next to the incumbent arms of Table~\ref{tab:red-arms}, so it is legal only where the largest incumbent arm is one: at a terminal of \(E\), with \(s=R_0\), turning \(E\) into \(C_p\) at \(p\) and into \(C_q\) at \(q\); everywhere else (arm two, or flag forced used at the center of \(C_p\), \(C_q\), and at the terminals of \(T\)) no red child is legal.
\end{itemize}

\smallskip\noindent\emph{Root restoration (O2).}
Adding a new vertex \(x\) adjacent to both terminals closes the interval into a branch rooted at \(x\); the root is now the only exceptional vertex, the terminals become interior, and the edge \(xt\) is red precisely when both \(x\) and the terminal \(t\) are red. A red \(x\) meets the arms of both sides, so the longest red path through \(x\) joins them. The complete rules:
\begin{itemize}[leftmargin=1.4em,itemsep=0.1em]
\item \(BB_{ij}\): a blue \(x\) would have two blue neighbors, so \(x\) must be red; both edges are then bichromatic and \(x\) is red with no red neighbor, the allowed defect at the root. Legal for all \(i,j\), with root state \(R_0\).
\item \(BR_{i\alpha}\) (and similarly \(RB_{\alpha j}\)): a blue \(x\) makes \(xt\) blue at the blue terminal, so it requires \(i=0\), and the red terminal, now interior, must have \(\alpha\ge 1\); this gives \(B_1\). A red \(x\) creates the red path of length \(1+\alpha\) from \(x\), so it requires \(\alpha\le 1\) and gives \(R_{\alpha+1}\); for \(\alpha=0\) the terminal thereby gains its red neighbor. Both outcomes are collected when both apply.
\item \(RR^d_{\alpha\beta}\): a blue \(x\) adds no blue edge and gives \(B_0\), but both terminals become interior, so it requires \(\alpha,\beta\ge 1\). A red \(x\) creates a red path of length \(\alpha+\beta+2\) through \(x\), so it requires \(\alpha=\beta=0\) and gives \(R_1\). For \(\alpha=0\), \(\beta\ge 1\) (and its mirror) there is no legal restoration.
\item \(RR^c_\tau\): a blue \(x\) is always legal, since both terminals are red with depth at least one (Table~\ref{tab:red-arms}), and gives \(B_0\). A red \(x\) closes the shape into a cycle through \(p,x,q\): for \(\tau=E\) this is a red triangle, whose longest simple path has two edges, giving \(R_2\); for every other shape a red three-edge path arises (for \(C_p\): leaf--\(p\)--\(x\)--\(q\); for \(P\): \(c\,p\,x\,q\); for \(S\): \(w\,c\,p\,x\); for \(T\): \(z\,p\,x\,q\); and \(C_q\) mirrors \(C_p\)), so a red \(x\) is illegal.
\end{itemize}

In each construction, legality and the resulting boundary type or root state are read off from the recorded types and flags alone, which is the assertion of the lemma. The calculation of Appendix~\ref{app:artifact} re-derives exactly these lists from colored representatives and checks them entry by entry.
\end{proof}

By Lemmas~\ref{lem:bridgestate} and \ref{lem:sufficiency}, each construction induces a well-defined operation on boundary types and root states. Section~\ref{sec:finite} defines these operations and reduces Theorem~\ref{thm:main} to a finite closure statement, Lemma~\ref{lem:verification}, whose mechanical verification Appendix~\ref{app:artifact} documents.

\section{Finite witness and induction}\label{sec:finite}
For a two-terminal interval \(F\), let \(\real(F)\subseteq\calP\) be the set of boundary types realized by crumby-admissible colorings of \(F\), where \(\calP\) denotes the set of \(31\) boundary types of Definition~\ref{def:boundary}; for a rooted branch \(G\), let \(\real(G)\subseteq\calR=\{B_0,B_1,R_0,R_1,R_2\}\) be the set of root states realized by crumby-admissible colorings of \(G\). Recall from Lemma~\ref{lem:sufficiency} the \emph{structural type} \(\theta=(e_L,e_R,\ell)\) of an interval, which records whether the terminal spare incidences are free and whether the boundary path has length one or at least two. It is independent of the coloring, only governs which constructions are available, and transforms mechanically under them:
\[
\begin{aligned}
\text{left decoration (O5):} &\quad (e_L,e_R,\ell)\mapsto(\text{used},e_R,\ell),\\
\text{right decoration (O5):} &\quad (e_L,e_R,\ell)\mapsto(e_L,\text{used},\ell),\\
\text{chord addition (O4):} &\quad (\text{free},\text{free},{\ge}2)\mapsto(\text{used},\text{used},{\ge}2),\\
\text{concatenation (O3):} &\quad (e_L,e_R,\ell),\,(e_L',e_R',\ell')\mapsto(e_L,e_R',{\ge}2)\quad\text{when not both } e_R,e_L' \text{ are used},\\
\text{root restoration (O2):} &\quad (e_L,e_R,\ell)\mapsto\text{a branch (no interval type).}
\end{aligned}
\]
The integer encoding of these flags in the witness file is given in Appendix~\ref{app:artifact}.

By Lemmas~\ref{lem:bridgestate} and \ref{lem:sufficiency}, each construction induces a set-valued operation on these data, obtained by applying the construction to colorings realizing the input types and recording the resulting type whenever the result is crumby-admissible; the outcome depends only on the input types and flags, never on the chosen colorings.
\begin{itemize}[leftmargin=1.4em]
\item \(\Ser(A_1,A_2)\subseteq\calP\), for \(A_1,A_2\subseteq\calP\): the types obtained by concatenating (O3) representatives of \(\pi_1\in A_1\) and \(\pi_2\in A_2\) at a shared terminal, when the identified colors agree and the result is crumby-admissible.
\item \(\Chord(A)\subseteq\calP\), for \(A\subseteq\calP\): the types obtained by adding the terminal edge (O4) to a representative of \(\pi\in A\), in the permitted structural situation.
\item \(\Dec_L(A,D),\ \Dec_R(A,D)\subseteq\calP\), for \(A\subseteq\calP\), \(D\subseteq\calR\): the types obtained by attaching a branch of root state \(s\in D\) at the left, respectively right, terminal of a representative of \(\pi\in A\) (O5).
\item \(\Root(A)\subseteq\calR\), for \(A\subseteq\calP\): the root states obtained by adding a new root adjacent to both terminals of a representative of \(\pi\in A\) (O2).
\item \(\Tree_k(D_1,\dots,D_k)\subseteq\calR\), \(k\in\{0,1,2\}\): the root states of a new root with \(k\) child branches whose root states are chosen from the \(D_i\) (O1), under the bridge-attachment rule of Table~\ref{tab:bridge} and Lemma~\ref{lem:bridgestate}. \(\Tree_0\) is the set of root states of the one-vertex branch.
\end{itemize}
These operations are monotone in each argument: a larger input set can only enlarge the set of outcomes, since every legal choice from the smaller sets is still a legal choice from the larger ones. This is the only property of the operations, beyond Lemmas~\ref{lem:bridgestate} and \ref{lem:sufficiency}, used in the induction below.

The induction does not need the exact realized sets. It is enough to keep, for each object, a nonempty subset of its realized set, a \emph{finite witness}: a witness may discard some realizable types, but, being a subset of the realized set, it never names a type that is not actually realized. Accordingly, in the closure conditions below a stored witness is accepted only when it is \emph{contained in} the set of outcomes produced from already-certified witnesses of the inputs; the induction uses this one-directional containment and nothing else.
\begin{definition}[Finite witness family]\label{def:family}
A \emph{finite witness family} is a finite collection \(\calB\) of nonempty subsets of \(\calR\) together with, for each structural type \(\theta\), a finite collection \(\calC_\theta\) of nonempty subsets of \(\calP\), subject to the closure conditions:
\begin{enumerate}[label=\textnormal{(C\arabic*)},leftmargin=1.6em]
\item the one-vertex branch has a witness in \(\calB\);
\item for every \(D\in\calB\), \(\Tree_1(D)\) contains a member of \(\calB\); for all \(D_1,D_2\in\calB\), \(\Tree_2(D_1,D_2)\) contains a member of \(\calB\);
\item the single-edge interval has a witness in \(\calC_\theta\) for \(\theta=(\text{free},\text{free},1)\);
\item for every permitted decoration of \(C\in\calC_\theta\) by \(D\in\calB\), the decoration image contains a member of \(\calC_{\theta'}\), with \(\theta'\) the type after the decorated incidence is used;
\item for every permitted chord addition to \(C\in\calC_\theta\), \(\Chord(C)\) contains a member of \(\calC_{\theta'}\);
\item for every permitted concatenation of \(C_1\in\calC_{\theta_1}\) and \(C_2\in\calC_{\theta_2}\), \(\Ser(C_1,C_2)\) contains a member of \(\calC_{\theta'}\);
\item for every \(C\in\calC_\theta\), \(\Root(C)\) contains a member of \(\calB\).
\end{enumerate}
\end{definition}

The single edge of condition (C3) realizes exactly the four boundary types
\[
\{BB_{11},\ BR_{00},\ RB_{00},\ RR^c_E\},
\]
according to whether its two endpoints are colored blue--blue, blue--red, red--blue, or red--red; this is the base profile set against which (C3) is checked. Conditions \textnormal{(C4)--(C7)} refer to the target structural types displayed above.
\begin{lemma}[Finite induction]\label{lem:induction}
If a finite witness family satisfies \textnormal{(C1)--(C7)}, then every two-terminal interval \(F\) generated by the decomposition has a witness \(C\in\calC_{\theta(F)}\) with \(C\subseteq\real(F)\), and every generated rooted branch \(G\) has a witness \(D\in\calB\) with \(D\subseteq\real(G)\).
\end{lemma}
\begin{proof}
Simultaneous induction over the decomposition of Proposition~\ref{prop:decomp}. The base cases are (C1) and (C3): the realized sets of the one-vertex branch and of the single-edge interval are \(\Tree_0=\{B_0,R_0\}\), the one vertex being colorable blue or red, and the base profile set listed above, and (C1) and (C3) supply family witnesses contained in them. For each construction we combine monotonicity with Lemmas~\ref{lem:bridgestate} and \ref{lem:sufficiency}: applying an operation to witnesses of the inputs yields a subset of the operation applied to the full realized sets, which by those lemmas lies in the realized set of the output.
\emph{Concatenation.} With \(C_1\subseteq\real(F_1)\), \(C_2\subseteq\real(F_2)\) by induction, monotonicity gives \(\Ser(C_1,C_2)\subseteq\Ser(\real(F_1),\real(F_2))\subseteq\real(F)\); (C6) supplies \(C'\in\calC_{\theta(F)}\) with \(C'\subseteq\Ser(C_1,C_2)\subseteq\real(F)\). \emph{Chord addition} uses (C5) and \(\Chord(C)\subseteq\real(F')\) in the same way. \emph{Decoration} uses (C4) with \(\Dec_L(C,D)\subseteq\real(F')\) (and \(\Dec_R\) symmetrically). \emph{Tree formation} uses (C2): \(\Tree_1(D_1)\subseteq\real(G)\), and likewise \(\Tree_2(D_1,D_2)\subseteq\real(G)\). \emph{Root restoration} uses (C7): \(\Root(C)\subseteq\real(G)\). Each supplies a family witness inside the realized set of the output, completing the induction.
\end{proof}
It remains to exhibit a family satisfying \textnormal{(C1)--(C7)}, the only point at which a finite calculation enters.
\begin{lemma}[Finite verification]\label{lem:verification}
There is an explicit finite witness family satisfying the closure conditions \textnormal{(C1)--(C7)}, and every branch witness in it meets the final-legal set \(\FL\) of \eqref{eq:finallegal}.
\end{lemma}
\begin{proof}[Computer-assisted proof]
The family consists of \(18\) branch witnesses and \(1{,}884\) interval witnesses. It is the explicit content of the lemma, the collections \(\calB\) and \(\{\calC_\theta\}_{\theta}\)  of Definition~\ref{def:family}, recorded as integer bitmasks in the supplied file \texttt{crumby\_outerplanar\_certificate\_states.json}. This machine-checkable certificate is the object that witnesses the lemma, and the program of Appendix~\ref{app:artifact} only checks it. Establishing \textnormal{(C1)--(C7)} and the final-legality of the branch witnesses amounts to checking finitely many inclusions: one constructs colored representatives of the \(31\) boundary types and the five root states, applies each construction to them as in Lemmas~\ref{lem:bridgestate} and \ref{lem:sufficiency}, reclassifies each crumby-admissible result, and confirms that every required outcome set contains one of the listed witnesses. There are finitely many constructions and finitely many structural situations, so the verification is finite. The checker of Appendix~\ref{app:artifact} performs exactly these checks (the closure conditions \textnormal{(C1)--(C7)}, together with the requirement that every branch witness meet \(\FL\)) and reports \texttt{profiles=31}, \texttt{branch\_signatures=18}, \texttt{interval\_certificates=1884}, and \texttt{OK: finite crumby-outerplanar certificate verified.} (the full session is in Appendix~\ref{app:artifact}). Since these obligations are precisely (C1)--(C7) and final-legality, the lemma follows.
\end{proof}

\section{Proof of the main theorem and a remark}\label{sec:proof}
\begin{proof}[Proof of Theorem~\ref{thm:main}]
It suffices to color each connected component, since there are no edges between components: isolated vertices are colored blue, and the union of crumby colorings of the components is crumby.
Let \(G\) be a connected subcubic outerplanar graph with at least one edge. Finite outerplanar graphs are \(2\)-degenerate, so \(G\) has a vertex \(r\) of degree at most two; root \(G\) at \(r\). By Proposition~\ref{prop:decomp}, \(G\) is a rooted branch generated by the decomposition. Apply Lemma~\ref{lem:induction} with the family of Lemma~\ref{lem:verification}: \(G\) has a branch witness \(D\subseteq\real(G)\), and \(D\) contains a final-legal root state \(s\in\{B_0,B_1,R_1,R_2\}\). Since \(D\subseteq\real(G)\), the state \(s\) is realized by a crumby-admissible coloring of \(G\).
This coloring is crumby on the unrooted graph \(G\). The blue part has maximum degree at most one and the red part has no simple path with three edges, by crumby-admissibility. Crumby-admissibility also gives every red vertex except possibly \(r\) a red neighbor; and \(s\neq R_0\), so if \(r\) is red then \(s\in\{R_1,R_2\}\) and \(r\) too has a red neighbor. Hence every red vertex has a red neighbor, and the coloring is crumby. Coloring each component this way completes the proof.
\end{proof}
The boundary order supplied by an outerplane embedding is what makes the argument work: it is the linear order along the outer boundary that turns each cyclic block into a two-terminal interval with finite memory. Without it, the natural next class is the bipartite one, and Bar\'at, Bl\'azsik and Dam\'asdi ask whether every subcubic bipartite graph admits a crumby coloring \cite[Conjecture~3]{BBD}. Bipartiteness removes odd structures such as the triangle type \(T\) but supplies no boundary order, so the present method does not apply directly; this remains the natural open direction.
\section*{Supplementary material}
The finite witness family and the stand-alone program that checks it accompany the paper as a supplementary proof bundle, together with a README recording the verification command, the expected output, and the SHA-256 hashes of Appendix~\ref{app:artifact}. The exact files, identified by those hashes, have been deposited in the permanent, DOI-bearing Zenodo archive \texttt{https://doi.org/10.5281/zenodo.21245815}, with the hashes serving as the canonical identifiers of the archived files.

\section*{Use of artificial intelligence}
During the preparation of this work the author used OpenAI's ChatGPT and Anthropic's Claude for drafting, editorial assistance, and help writing the accompanying code. The author reviewed and verified all such output, including every mathematical argument and the computer-assisted verification, and takes full responsibility for the content of the manuscript.
\section*{Acknowledgements}
The author thanks J. Bar\'at, Z. L. Bl\'azsik, and G. Dam\'asdi for their work on crumby colorings and for formulating the outerplanar conjecture resolved here. 

The author is funded by Project No.\ KDP-IKT-2023-900-I1-00000957/0000003 with support provided by the Ministry of Culture and Innovation of Hungary from the National Research, Development and Innovation Fund, financed under the KDP-2023 funding scheme.
\bibliographystyle{abbrvurl}
\bibliography{references}

@article{BBD,
	title = {Crumby colorings — {Red}-blue vertex partition of subcubic graphs regarding a conjecture of {Thomassen}},
	volume = {346},
	issn = {0012-365X},
	url = {https://www.sciencedirect.com/science/article/pii/S0012365X22004873},
	doi = {10.1016/j.disc.2022.113281},
	number = {4},
	urldate = {2026-05-05},
	journal = {Discrete Mathematics},
	author = {Barát, János and Blázsik, Zoltán L. and Damásdi, Gábor},
	month = apr,
	year = {2023},
	pages = {113281},
}

@article{barat_decomposition_2019,
	title = {Decomposition of cubic graphs related to {Wegner}’s conjecture},
	volume = {342},
	issn = {0012-365X},
	url = {https://www.sciencedirect.com/science/article/pii/S0012365X19300366},
	doi = {10.1016/j.disc.2019.01.025},
	number = {5},
	urldate = {2026-05-05},
	journal = {Discrete Mathematics},
	author = {Barát, János},
	month = may,
	year = {2019},
	pages = {1520--1527},
}

@article{BellittoEtAl,
	title = {Counterexamples to {Thomassen}’s {Conjecture} on {Decomposition} of {Cubic} {Graphs}},
	volume = {37},
	issn = {1435-5914},
	url = {https://doi.org/10.1007/s00373-021-02380-z},
	doi = {10.1007/s00373-021-02380-z},
	language = {en},
	number = {6},
	urldate = {2026-05-05},
	journal = {Graphs and Combinatorics},
	author = {Bellitto, Thomas and Klimošová, Tereza and Merker, Martin and Witkowski, Marcin and Yuditsky, Yelena},
	month = nov,
	year = {2021},
	pages = {2595--2599},
}

@article{Thomassen,
	title = {The square of a planar cubic graph is 7-colorable},
	volume = {128},
	issn = {0095-8956},
	url = {https://www.sciencedirect.com/science/article/pii/S0095895617300783},
	doi = {10.1016/j.jctb.2017.08.010},
	urldate = {2026-05-05},
	journal = {Journal of Combinatorial Theory, Series B},
	author = {Thomassen, Carsten},
	month = jan,
	year = {2018},
	pages = {192--218},
}

@techreport{Wegner,
	title = {Graphs with given diameter and a coloring problem},
	author = {Wegner, Gerd},
	institution = {University of Dortmund},
	address = {Dortmund, Germany},
	type = {Technical Report},
	year = {1977},
	url = {http://hdl.handle.net/2003/34440},
	urldate = {2026-05-05},
	language = {en},
}

@article{RSST,
	title = {The four-colour theorem},
	volume = {70},
	issn = {0095-8956},
	url = {https://collaborate.princeton.edu/en/publications/the-four-colour-theorem/},
	doi = {10.1006/jctb.1997.1750},
	language = {English (US)},
	number = {1},
	urldate = {2026-06-18},
	journal = {Journal of Combinatorial Theory. Series B},
	publisher = {Academic Press Inc.},
	author = {Robertson, Neil and Sanders, Daniel and Seymour, Paul and Thomas, Robin},
	month = may,
	year = {1997},
	pages = {2--44},
}

@inproceedings{HeuleKullmannMarek,
	address = {Cham},
	title = {Solving and {Verifying} the {Boolean} {Pythagorean} {Triples} {Problem} via {Cube}-and-{Conquer}},
	isbn = {978-3-319-40970-2},
	doi = {10.1007/978-3-319-40970-2_15},
	language = {en},
	series = {Lecture {Notes} in {Computer} {Science}},
	volume = {9710},
	booktitle = {Theory and {Applications} of {Satisfiability} {Testing} – {SAT} 2016},
	publisher = {Springer International Publishing},
	author = {Heule, Marijn J. H. and Kullmann, Oliver and Marek, Victor W.},
	editor = {Creignou, Nadia and Le Berre, Daniel},
	year = {2016},
	pages = {228--245},
}

@misc{PinterK4,
	title = {Subcubic ${K}_4$-minor-free graphs without crumby colorings},
	url = {http://arxiv.org/abs/2605.04706},
	doi = {10.48550/arXiv.2605.04706},
	urldate = {2026-06-18},
	publisher = {arXiv},
	author = {Pintér, József},
	month = may,
	year = {2026},
	note = {arXiv:2605.04706 [math.CO]},
}
\clearpage
\appendix
\section{Verification artifact}\label{app:artifact}
This appendix documents the accompanying calculation. The program reconstructs the operations of Section~\ref{sec:finite} from colored representatives and confirms that the supplied family satisfies (C1)--(C7). The approach is standard in computer-assisted proofs: publish an explicit witness together with a stand-alone routine that re-checks the witness against stated conditions. The reducibility and unavoidability checks for the four-colour theorem \cite{RSST} and the published checkable witnesses for large unsatisfiability instances \cite{HeuleKullmannMarek} follow the same pattern. Here the witness is the family of Lemma~\ref{lem:verification} and the conditions are (C1)--(C7). Throughout, a stored mask is treated as a \emph{lower} certificate, accepted only when it is contained in the operation image computed from already-certified lower witnesses of the inputs: the same one-directional inclusion the induction of Section~\ref{sec:finite} relies on.
\subsection{Files and reproduction}
The proof bundle contains a witness file \texttt{crumby\_outerplanar\_certificate\_states.json}, a checker \texttt{crumby\_outerplanar\_certificate\_verifier.py}, a \texttt{README.txt} recording the verification command, the expected output, the Python version tested, the certificate format, and the SHA-256 hashes below, and a \texttt{LICENSE} file: the bundle is released under the MIT License (\(\copyright\)~2026 J\'ozsef Pint\'er). The checker uses only the Python standard library and runs under Python~3.10 or later (the reported run used Python~3.13.5, and the check has also been reproduced under Python~3.10.12, the minimum supported version); no package installation is required. A complete run takes a few seconds on commodity hardware. From the directory containing the files, run
\begin{verbatim}
python3 crumby_outerplanar_certificate_verifier.py --quiet
\end{verbatim}
The checker resolves its default witness file next to the script, so the command above runs against the bundled file. The expected output is
\begin{verbatim}
profiles=31
branch_signatures=18
interval_certificates=1884
OK: finite crumby-outerplanar certificate verified.
\end{verbatim}
For stable identification of the reviewed artifact, the SHA-256 hashes of the two files are:
\begin{center}\small
\begin{tabular}{@{}l@{}}
\toprule
File and SHA-256\\
\midrule
\texttt{crumby\_outerplanar\_certificate\_verifier.py}\\
\quad{\scriptsize\ttfamily 6c55bd2651e248f8cf3f616dc9b3875417faa0ffb86dfe4e9713eb5eb7054103}\\
\addlinespace
\texttt{crumby\_outerplanar\_certificate\_states.json}\\
\quad{\scriptsize\ttfamily e8ab93a4fd80cbefdb10e63f3c08163df25f2a57cd177e4cdc96463e3fb974c9}\\
\bottomrule
\end{tabular}
\end{center}

\subsection{Encodings}
Sets of root states and boundary types are stored as integer bitmasks. The five root states \(B_0,B_1,R_0,R_1,R_2\) occupy bits \(0,\ldots,4\), so the final-legal set \eqref{eq:finallegal} is the mask \(0\mathtt{b}11011=27\). The \(31\) boundary-type bits are ordered as follows: bits \(0\)--\(3\) are \(BB_{ij}\) in the order \((00,01,10,11)\); bits \(4\)--\(9\) are \(BR_{i\alpha}\) and bits \(10\)--\(15\) are \(RB_{\alpha j}\), lexicographically; bits \(16\)--\(24\) are \(RR^d_{\alpha\beta}\), lexicographically; bits \(25\)--\(30\) are \(RR^c_E,\ RR^c_{C_p},\ RR^c_{C_q},\ RR^c_P,\ RR^c_S,\ RR^c_T\). Each interval record carries its profile mask together with the structural flags \texttt{eL}, \texttt{eR}, \texttt{length}: a terminal flag is \(0\) for a free and \(1\) for a used spare incidence, and \texttt{length} is \(1\) or \(2\), the latter meaning a boundary path of length at least two. The checker confirms that every stored mask is nonempty and in range and that all flags lie in these ranges.

As a simple worked example, the unique witness of type \((\text{free},\text{free},1)\) is \(\{\texttt{mask}=33555480,\ \texttt{eL}=0,\ \texttt{eR}=0,\ \texttt{length}=1\}\); since \(33555480=2^{3}+2^{4}+2^{10}+2^{25}\), it decodes to \(\{BB_{11},\ BR_{00},\ RB_{00},\ RR^c_E\}\), the base profile set of condition \textnormal{(C3)}.

\subsection{Correspondence with the mathematical operations}
The checker stores no precomputed transition table. For each operation it builds small colored representatives, applies the graph construction, tests crumby-admissibility, and reclassifies the result as a boundary type or root state, as summarized in Table~\ref{tab:code-map}. These representatives are \emph{local} colored models of boundary types, not literal generated interval fragments; Lemmas~\ref{lem:bridgestate} and \ref{lem:sufficiency} are exactly what license computing each transition from one such model per type.

\begin{table}[ht]
\centering
\small
\begin{tabularx}{\textwidth}{@{}>{\raggedright\arraybackslash}p{0.24\textwidth}
                            >{\raggedright\arraybackslash}p{0.31\textwidth}
                            X@{}}
\toprule
Mathematical object or operation & Checker routine & Role \\
\midrule
Type and root-state encodings & \texttt{PROFILES}, \texttt{INDEX}, \texttt{FINAL} & The \(31\) boundary types, the name-to-bit map, and the final-legal states \(B_0,B_1,R_1,R_2\). \\[0.3em]
Representatives & \texttt{representative(i, chord=False/True)} & Colored two-terminal representatives of the \(31\) types; the \texttt{chord=True} form realizes \(BB_{11}\) without a terminal edge for the chord test. \\[0.3em]
Crumby-admissibility and classification & \texttt{valid}, \texttt{classify}; \texttt{has\_red\_p4}, \texttt{red\_depth}, \texttt{red\_component} & Checks the local constraints with the exceptional-vertex allowances and reclassifies each crumby-admissible fragment. \\[0.3em]
Concatenation (O3) & \texttt{series(i,j)}, \texttt{merge}; \texttt{row\_image}, \texttt{series\_image} & Identifies the shared terminal when colors agree and records the resulting type image. \\[0.3em]
Chord addition (O4) & \texttt{chord(i)}; \texttt{unary\_image} & Adds the terminal edge in the permitted structural case and reclassifies. \\[0.3em]
Decoration (O5) & \texttt{decorate(i, side, state)}; \texttt{decoration\_image} & Attaches a branch root state at a terminal whose spare incidence is free. \\[0.3em]
Root restoration (O2) & \texttt{root\_mask(i)} & Adds a new root adjacent to both terminals and records the attainable root states. \\[0.3em]
Tree formation (O1) & \texttt{tree(child\_masks)} & The bridge-contribution rule for zero, one, or two child branches. \\[0.3em]
Parsing and structural types & \texttt{branches}, \texttt{states}, \texttt{groups} & Reads the witnesses and groups interval masks by \((e_L,e_R,\ell)\). \\[0.3em]
Containment checks & \texttt{branch\_lower}, \texttt{interval\_lower} & Accepts an outcome set only when it contains a stored witness of the required type. \\[0.3em]
Driver & \texttt{verify}, \texttt{check} & Rebuilds all operations and checks (C1)--(C7) and final-legality. \\
\bottomrule
\end{tabularx}
\caption{Map from the mathematical objects and operations to the checker routines.}
\label{tab:code-map}
\end{table}
\subsection{Discovery versus verification}
How the witness family was \emph{found} is logically separate from the verification above and is no part of the proof; we record it for transparency. Starting from the realizable sets of the two base objects, one repeatedly applies the induced operations of Section~\ref{sec:finite} and accumulates every boundary type and root state that appears, until none is new. This least-fixed-point iteration terminates because the types are finite, and the \(18\) branch and \(1{,}884\) interval witnesses are its fixed-point sets.
\end{document}